\documentclass{amsart}

\usepackage{amsmath, srcltx}
\usepackage{amssymb}
\usepackage{amsfonts,cmtiup,comment}

\usepackage{mathrsfs,cmtiup}
\usepackage{amsthm, eucal, eufrak}

\makeatletter
\def\blfootnote{\xdef\@thefnmark{}\@footnotetext}
\makeatother

\newtheorem{theorem}{Theorem}[section]
\newtheorem{lemma}[theorem]{Lemma}
\newtheorem{proposition}[theorem]{Proposition}
\newtheorem{corollary}[theorem]{Corollary}

\theoremstyle{remark}

\let\le=\leqslant
\let\ge=\geqslant
\let\leq=\leqslant
\let\geq=\geqslant

\numberwithin{equation}{section}

\begin{document}

\title[Words and pronilpotent subgroups in profinite groups]{Words and pronilpotent subgroups\\ in profinite groups}

\author{E. I. Khukhro}

\address{Sobolev Institute of Mathematics\\ Novosibirsk, 630\,090,
Russia\\ \emph{Email address:} {khukhro@yahoo.co.uk}}
\author{P. Shumyatsky}
\address{Department of Mathematics\\ University of Brasilia\\ DF~70910-900, Brazil\\
\emph{Email address:} {pavel@unb.br}}


\begin{abstract}
Let $w$ be a multilinear commutator word, that is, a commutator of weight $n$ in $n$ different group variables. It is proved that if $G$ is a profinite group in which all pronilpotent subgroups generated by $w$-values are periodic, then the verbal subgroup $w(G)$ is locally finite.
\end{abstract}

\maketitle
\blfootnote{American Mathematical Society 2010 MSC: Primary 20E18; Secondary 20D10, 20D20, 20F50\\ 

This work was supported by CNPq-Brazil. The first author thanks  CNPq-Brazil and the University of Brasilia for support and hospitality that he enjoyed during his visits to Brasilia.}

\section{Introduction}

By a \emph{multilinear commutator word} we mean a commutator of some weight $n$ in $n$ different group variables.
Such words (also known as outer commutator words) 
are obtained by nesting commutators, but using always
different variables. For example, the word $[[[x_1, x_2], [[x_3, x_4], x_5]], x_6]$ is a multilinear
commutator, while the Engel word $[[[x_1, x_2], x_2], x_2]$ is not. For any multilinear commutator word $w$ the law $w=1$ defines a soluble variety of groups. Among well-known multilinear commutator  words there are simple commutators $[x_1, \dots  , x_k]$, for which the corresponding verbal subgroups are the terms of the lower central series. Another distinguished sequence of multilinear commutators is formed by
the derived words
which are defined recursively as  $\delta _0 = x_1$, $\delta _k = [\delta _{k-1}(x_1,\dots  , x_{2^{k-1}}),\, \delta _{k-1}(x_{2^{k-1}+1},\dots  , x_{2^k} )]$; the corresponding verbal subgroups are the terms of the derived series.

The purpose of the present paper is to prove the following theorem.

\begin{theorem}\label{main}
Let $w$ be a multilinear commutator word.
Let $G$ be a profinite group in which all pronilpotent subgroups generated by $w$-values are periodic. Then the verbal subgroup $w(G)$ is locally finite.
\end{theorem}

Theorem~\ref{main} can be regarded as a generalization of the theorem on local finiteness of periodic profinite groups, which corresponds to the case of $w=x_1$ being a group variable. That theorem was proved in Zelmanov's paper \cite{zel92} on pro-$p$ groups, while the reduction to pro-$p$ groups was obtained by Wilson \cite{wil83} (using the classification of finite simple groups).
The proof of Theorem~\ref{main} also uses the classification and Zelmanov's theorem. If the hypothesis on all pronilpotent subgroups generated by $w$-values being periodic is replaced by being locally finite, then the proof of Theorem~\ref{main} becomes independent of Zelmanov's theorem; in this sense it is a generalization of Wilson's theorem.

The case of Theorem~\ref{main} where $w=[x,y]$ was earlier proved in \cite{dms2}.

Theorem~\ref{main} is related to the conjecture that if, for some group word $w$,  all $w$-values in a profinite group $G$ have finite order, then the  verbal subgroup $w(G)$ must be locally finite (cf. special cases 
stated in \cite{shu01} and similar problems 15.104 and 17.126 in Kourovka Notebook \cite{kour}).  This conjecture is a natural generalization of the aforementioned Zelmanov's theorem on local finiteness of periodic profinite groups. Note that in this conjecture the orders of $w$-values are not assumed to be bounded.  Also  in Theorem~\ref{main} the exponents of pronilpotent subgroups generated by $w$-values are not assumed to be bounded, just like  in the theorems of Wilson and Zelmanov periodic profinite groups are not assumed to be of bounded exponents. There is, however, a well-known 
conjecture that a periodic profinite group must in fact have bounded exponent.

There is a similar conjecture for finite groups, which of course must involve bounds for the exponent: if a finite group $G$ satisfies a law $w^e=1$ for some group word $w$, then the exponent of the verbal subgroup $w(G)$ must be bounded in terms of $w$ and $e$ only.
In relation to this problem it was proved in \cite{austral} that if, for a multilinear commutator $w$,  all nilpotent subgroups in a finite group $G$ satisfy the law $w^e=1$, then the exponent of the verbal subgroup  $w(G)$ is bounded in terms of $w$ and $e$ only. A~similar result was also proved in \cite{dms1} for the case of $w$ being an Engel word and for some other cases.

 The proof of Theorem~\ref{main}  relies on several reductions, some of which are firstly performed for finite groups. In order to state the corresponding results we introduce some notation. As in \cite{khu-shu13} we introduce the \emph{non-$p$-soluble length} $\lambda _p(G)$ of a finite group $G$ as the minimum 
 number of non-$p$-soluble factors in a normal series each of whose factors either is $p$-soluble or is a direct product of non-abelian simple groups of order divisible by~$p$. For $p=2$ we naturally use the term \emph{nonsoluble length} and denote it by $\lambda (G) =\lambda _2(G)$. Bounding the non-$p$-soluble length is an important tool in reductions to (pro-)$p$-soluble groups, for example, implicitly in the Hall--Higman paper \cite{ha-hi}, or in the aforementioned Wilson's paper \cite{wil83}. In this paper we obtain bounds for non-$p$-soluble length using, in particular,  our recent results in \cite{khu-shu13}.

 Throughout the paper, $w$ is a multilinear commutator word of weight~$n$. For a (profinite) group $G$ we denote by $G_w$ the set of all values of $w$ on elements of~$G$. Clearly, $G_w$ is a normal subset of $G$; 
 it is easy to see that $G_w$ is also closed under taking  
 inverses. We denote by $w(G)=\langle G_w\rangle$ the corresponding (closed) verbal subgroup. 
If $P$ is a subgroup of a (profinite)
group $G$, we denote by $W_G(P)$ the (closed) subgroup of $P$ generated by all elements of $P$ that are conjugate in $G$ to $w$-values on elements of $P$, that is, $W_G(P)=\langle P_w^{\;G}\cap P\rangle$. When it causes no confusion, we will write $W(P)$ in place of $W_G(P)$.

\begin{theorem}\label{gera}
Let $H$ be a normal subgroup of a finite group $G$, and let $P$ be a Sylow $p$-subgroup of~$H$. Suppose that for some $t\in G $
 the coset $tW_G(P)$ is of exponent dividing~$p^a$. Then the non-$p$-soluble length of $H$ is bounded in terms of $a$ and $n$ only.
 \end{theorem}

We also need bounds for the $p$-length of $p$-soluble subgroups.

\begin{theorem}\label{solu}
Let $H$ be a normal $p$-soluble subgroup of a finite group $G$, and let $P$ be a Sylow $p$-subgroup of~$H$. Suppose that for some $t\in G$ the coset $tW_G(P)$ is of exponent dividing $p^a$. Then the $p$-length of $H$ is bounded in terms of $a$ and~$n$.
\end{theorem}

The special cases of Theorems~\ref{gera} and~\ref{solu} for $w=x_1$, when $W_G(P)=P$,  were proved in \cite{wil83} for $p\ne 2$; in particular, Theorems~\ref{gera} and~\ref{solu} extend the results of Theorems~2* and 3* in \cite{wil83} to the case of $p=2$.

Apart from applications in the present paper, Theorem~\ref{solu} has a corollary giving an affirmative answer to a special case of Wilson's problem 9.68 in Kourovka Notebook \cite{kour} for a (finite exponent)-by-nilpotent variety.

 \begin{corollary}\label{c-exp-nil}
Let $\frak V$ be a group variety that is a product of  a variety of finite exponent and a nilpotent variety. Then the $p$-length of finite $p$-soluble groups whose Sylow $p$-subgroups belong to $\frak V$ is bounded. In other words, if a Sylow $p$-subgroup $P$ of a finite $p$-soluble group $G$ is such that
$(\gamma _{c+1}(P))^{p^a}=1$, then the $p$-length of $G$ is bounded in terms of $c$ and $a$ only.
 \end{corollary}

A result similar to Corollary~\ref{c-exp-nil} is also true for a nilpotent-by-(finite exponent) variety --- this case has been known in folklore as a consequence of the Hall--Higman theorems.

 We call pro-(finite soluble) groups, that is, inverse limits of finite soluble groups, simply \emph{prosoluble} groups for short; \emph{pro-$p$-soluble} groups are defined similarly. For a profinite group $G$ we denote
by $\pi (G)$ the set of prime divisors of the orders of elements of $G$ (understood as supernatural, or Steinitz, numbers). If a profinite group $G$ has $\pi (G)=\pi$, then we say that $G$ is a pro-$\pi$ group. Recall that Sylow theorems hold for $p$-Sylow subgroups of a profinite group (see, for example, \cite[Ch.~2]{wil-book}). When dealing with profinite groups we consider only continuous homomorphisms and quotients by closed normal subgroups.

Theorem~\ref{gera} has the following consequence for profinite groups.

\begin{theorem}\label{profi} Let $G$ be a profinite group, and $P$ a $p$-Sylow subgroup of~$G$. Suppose that $W_G(P)$ is locally finite. Then $G$ has a finite series of closed characteristic  subgroups in which each factor either is pro-$p$-soluble or is isomorphic to a Cartesian product of non-abelian finite simple groups of order divisible by~$p$.
\end{theorem}

In view of Zelmanov's theorem \cite{zel92} we could write that $W_G(P)$ is periodic instead of locally finite; this remarks applies throughout the paper.

Theorem~\ref{solu} has the following consequence for profinite groups.

\begin{theorem}\label{pro-p-solu}
Let $p$ be a prime and let $G$ be a pro-$p$-soluble group. Suppose that $W_G(P)$ is locally finite for a $p$-Sylow subgroup $P$ of~$G$. Then $G$ has a finite series of closed characteristic subgroups in which each factor is either a pro-$p$ group or a pro-$p'$ group.
\end{theorem}

If $G$ is a finite soluble group, let $h(G)$ denote the Fitting height of~$G$. Recall that this is the length of a shortest normal (or characteristic) series of $G$ all of whose factors are nilpotent. For a prosoluble group $G$ we define an analogue of Fitting height denoted by $ph(G)$ as the length of a shortest series of (closed) characteristic subgroups all of whose factors are pronilpotent; if $G$ has no finite series of this kind, then we write $ph(G)=\infty$. Of course, $ph(G)<\infty$ if and only if $G$ is an inverse limit of finite soluble groups of bounded Fitting height.

\begin{corollary}\label{prosolu}
Let $G$ be a prosoluble group such that $\pi(G)$ is finite. Suppose that $W_G(P)$ is locally finite for any $p$-Sylow subgroup of $G$, for any $p\in\pi(G)$. Then $ph(G)$ is finite.
\end{corollary}
\medskip

\section{Bounding non-$p$-soluble length of finite groups}
\label{s-nsl}

Our purpose in this section is to prove Theorem~\ref{gera}. We begin with quoting  a
special case of our recent result on non-$p$-soluble length in \cite{khu-shu13}. The definitions of non-$p$-soluble length $\lambda_p (G)$ and nonsoluble length $\lambda (G)= \lambda_2 (G)$ of a finite group $G$ were given in the Introduction.

 \begin{lemma}[see {\cite[Theorem~1.4]{khu-shu13}}]\label{t-exp-sol}
 Suppose that a Sylow $p$-subgroup of a finite group $G$ is soluble of derived length~$d$. Then the non-$p$-soluble length $\lambda _p(G)$ is bounded in terms of~$d$.
 \end{lemma}

 (In fact, Theorem~1.4 in \cite{khu-shu13} gives a bound for $\lambda _p(G)$ for finite groups $G$ whose Sylow $p$-subgroups belong to any given variety that is a product of soluble varieties and varieties of finite exponent, but we need only this special case here.)

By an \emph{almost simple group} we mean a finite group with a unique minimal normal subgroup that is a non-abelian simple group. By Schreier's conjecture confirmed by the classification, the quotient by that minimal normal subgroup is soluble. The following lemma is well known.

\begin{lemma}\label{l-alm-sim}
Suppose that a finite group $G$ has no nontrivial soluble normal subgroups. Let $M$  be the product of all minimal normal subgroups, and let $M=S_1\times\dots\times S_m$, where the $S_i$ are non-abelian simple groups. Then the kernel $K=\bigcap N_G(S_i)$ of the permutational action of $G$ on the set $\{ S_1, \dots , S_m\}$ embeds in a direct product of almost simple groups.
\end{lemma}

\begin{proof}
Clearly, $K/C_K(S_i)$ is an almost simple group for every~$i$. We have $\bigcap C_K(S_i)=1$ because $G$ has no nontrivial soluble normal subgroups. The result follows, since $K/\bigcap C_K(S_i)$ embeds
in the direct product of the groups $K/C_K(S_i)$.
\end{proof}

We now list three elementary lemmas
 that are used throughout the paper.

\begin{lemma}\label{l-w} Let $G$ be a (pro)finite group.

{\rm (a)} If $G_1\leq G$ and $P_1\leq P$, then $W_{G_1}(P_1)\leq W_G(P)$.

{\rm (b)} If $P$ is a $p$-Sylow subgroup of $G$ and $N$ is a (closed) normal subgroup of $G$, then $W_{G/N}(PN/N)= W_G(P)N/N$.
\end{lemma}

\begin{proof}
(a) This obviously follows from the definition, as the generating set for $W_{G_1}(P_1)$ is contained in that for $W_G(P)$.

(b) Let the bar denote images in $\bar G=G/N$. The inclusion $W_{\bar G}(\bar P)\geq \overline{W_G(P)}$ is obvious. An arbitrary generating element of $W_{\bar G}(\bar P) $ has the form $w(\bar a_1,\dots )^{\bar g}\in \bar P$ for $\bar a_i\in \bar P$ and $\bar g\in \bar G$. Taking arbitrary pre-images $a_i\in P$ and $g\in G$, we obtain a pre-image $w(a_1,\dots )^{ g}\in PN$. Since $P$ is a $p$-Sylow subgroup of $PN$, there is $h\in N$ such that $x=w(a_1,\dots )^{ gh}\in P$, which means that $x\in W_G(P)$ and obviously $\bar x=w(\bar a_1,\dots )^{\bar g}$. Hence, $W_{\bar G}(\bar P)\leq \overline{W_G(P)}$.
\end{proof}

\begin{lemma}[{\cite[Lemma 2.3]{mz}}]\label{lbw}
Let $G$ be a group, and $w$ a multilinear commutator word.
Let $N$ be a normal subgroup of $G$, and let $M$ be the subgroup generated by the values of $w$ that belong to~$N$. Then $[N,w(G)]\leq M$.
\end{lemma}

This lemma was proved in \cite{mz} for the case $M=1$, to which the proof immediately reduces by considering $G/M$.

Recall that $\delta_1=[x_1,x_2]$ and by induction
$$\delta_{k+1}=
[\delta_k(x_1,\dots, x_{2^k}),\, \delta_k(x_{2^k+1},\dots, x_{2^{k+1}})]
$$
are multilinear commutator words such that the law $\delta_d=1$ defines the variety of soluble groups of derived length~$d$.
 The next lemma is well-known (see for example \cite[Lemma~4.1]{S2}).

\begin{lemma}\label{delta}
Let $G$ be a group and let $w$ be a multilinear commutator word of weight~$n$. Then every $\delta_n$-value is a $w$-value.
\end{lemma}

We now embark on proving Theorem~\ref{gera}, beginning with the following key proposition.

\begin{proposition}\label{4}
Let $w$ be a multilinear commutator.
Let $P$ be a Sylow $p$-subgroup of a finite group $G$ and suppose that $G$ contains an element $t$ such that the coset $tW(P)$ is of exponent $p^e\neq1$. Assume further that $G$ has no nontrivial normal $p$-soluble subgroups. Then $G$ possesses two normal subgroups $K_1$ and $K_2$ such that their intersection $K=K_1\cap K_2$ embeds in a direct product of almost simple groups, the image of $tW(P)$ in $G/K_1$ is of exponent dividing $p^{e-1}$, and $W(P)\leq K_2$. In particular, if $e=1$, then $tW(P)\leq K_1$ and therefore also $W(P)\leq K$.
\end{proposition}

\begin{proof} Let $V$ be a minimal normal subgroup of~$G$. Clearly,
$V= S_1\times S_2\times\cdots\times S_r $, where $ S_1,S_2,\dots,S_r$
are isomorphic simple groups of order divisible by~$p$. Then $P_i=P\cap S_i$ is a Sylow $p$-subgroup of $S_i$ for $i=1,\dots,r$. Acting on $V$ by conjugation the group $G$ permutes the simple factors, so we obtain a representation of $G$ by permutations of the set $\left\{S_1,S_2,\dots,S_r\right\}$. Let $K_V$ be the kernel of this
representation.

Suppose that $W(P)\not\leq K_V$; we claim that then the image of $tW(P)$ in $G/K_V$ is of exponent dividing $p^{e-1}$. If this is false, then there exists $u\in tW(P)$ which has an orbit of length $p^e$ in the set $\left\{S_1,S_2,\dots,S_r\right\}$. Without loss of generality we assume that this orbit is $\left\{S_1,S_2,\dots,S_{p^e}\right\}$.

Since $W(P)\not\leq K_V$, there is a conjugate $g$ of an element of $P_w$
such that $g\not\in K_V$. Since $K_V$ is a normal subgroup, we can assume that in fact $g\in P_w$ and $g\not\in K_V$. Choose $j$ such that $S_j\not=S_j^g$ and let $b\in P_j$. By Lemma~\ref{lbw} the commutator $[b,g]$ is a product of elements of $P_w\cap V$. Since $[b,g]=b^{-1}b^g$, where $b\in S_j$ and $b^g\in S_j^g\ne S_j$, we deduce that $b$ can be written as a product of the $S_j$-projections of elements of $P_w\cap V$. Since $b$ was an arbitrary element of $P_j$, we conclude that $P_j$ is generated by the $S_j$-projection of $P_w\cap V$. Since $G$ permutes $S_1,S_2,\dots,S_r$ transitively, there is an element $q\in G$ such that $P_1=P_j^q$, so that $P_1$ is generated by the $S_1$-projection of $P_w^q\cap V$. Furthermore, for each $k=2,\dots, r$, using if necessary conjugations by $l_k\in S_k$ and replacing $P_w^q$ by $P_w^{ql_k}$ we can make sure that the $S_k$-projection of $P_w^{ql_k}$ is contained in $P_k$ (keeping the $S_1$-projection the same). Thus, after changing notation, we obtain that $P_1$ is generated by the $S_1$-projection of $P_w^q\cap V$ for some $q\in G$ such that $P_w^q\cap V\subseteq P$. The latter also means that $P_w^q\cap V\subseteq W(P)$.

Set $N_1=N_{S_1}(P_1)$. The well-known Burnside theorem on normal $p$-com\-ple\-ments \cite[Theorem 7.4.3]{go} shows that $P_1$ is not contained in the center $Z(N_1)$. Since $P_1$ is generated by the $S_1$-projection of $P_w^q\cap V$, we can choose $x\in P_w^q\cap V\subseteq W(P)$ and an element $h\in N_1$ such that $x^h\neq x$.

Now write $x=x_1\cdots x_r$, where $x_i\in P_i$. Of course, $x^h=x_1^hx_2\cdots x_r$. Since $ux\in tW(P)$, we have $(ux)^{p^e}=1$, in particular, the $S_1$-projection of $(ux)^{p^e}$ is trivial. A~direct calculation shows that the $S_1$-projection of $(ux)^{p^e}$ equals the product
$$
x_1x_2^{u^{p^e-1}}\cdots x_{p^e-1}^{u^2}x_{p^e}^u.
$$
Therefore $x_1x_2^{u^{p^e-1}}\cdots x_{p^e-1}^{u^2}x_{p^e}^u=1$.
Since the element $x^h$ belongs to $P_w^{qh}\cap V$ and to $P$, we have $x^h\in W(P)$. Therefore we can apply the same argument to the element $x^h$ in place of $x$, by which also $x_1^hx_2^{u^{p^e-1}}\cdots x_{p^e-1}^{u^2}x_{p^e}^u=1$. Hence we deduce that $x_1^h=x_1$. This is a contradiction with the choice of $h$ such that $x^h\neq x$.

Thus, indeed, the image of $tW(P)$ in $G/K_V$ is of exponent dividing $p^{e-1}$ whenever $W(P)\not\leq K_V$. It remains to set $K_1=\bigcap K_V$, where the intersection is taken over all minimal normal subgroups $V$ of $G$ such that $W(P)\not\leq K_V$, and $K_2=\bigcap K_U$, where $U$ runs over all minimal normal subgroups $U$ of $G$ such that $W(P)\leq K_U$. Then $K=K_1\cap K_2$ embeds in a direct product of almost simple groups by Lemma~\ref{l-alm-sim}.
\end{proof}

\begin{proof}[Proof of Theorem~\ref{gera}] Recall that $w$ is a multilinear commutator of weight $n$, $G$ is a finite group, $H$ a normal subgroup of $G$, $P$ a Sylow $p$-subgroup of $H$, and for some $t\in G$ the coset $tW_G(P)$ is of exponent dividing $p^a$. We need to show that $H$ has non-$p$-soluble length bounded in terms of $a$ and~$n$. Clearly, we can assume without loss of generality that $G=H\langle t\rangle$, and then our task is equivalent to bounding the non-$p$-soluble length of~$G$.

By Lemma~\ref{delta} every $\delta_n$-value is a $w$-value, that is, $P_w\supseteq P_{\delta _n}$. Since $P$ is a normal subgroup with cyclic quotient in a Sylow $p$-subgroup $\hat P$ of $G=H\langle t\rangle$, straightforward induction on $k$ also shows that $P_{\delta _{k}}\supseteq \hat P _{\delta _{k+1}}$ for all~$k$. Let temporarily $\hat w=\delta _{n+1}$. Then for the corresponding subgroup we have
\begin{align*}\hat W_G(\hat P)
&=\langle \hat P_{\delta _{n+1}}^{\;G}\cap \hat P\rangle
=\langle \hat P_{\delta _{n+1}}^{\;G}\cap \hat P\cap H\rangle\\
&= \langle \hat P_{\delta _{n+1}}^{\;G}\cap P\rangle
\leq \langle P_{\delta _{n}}^{\;G}\cap P\rangle
\leq \langle P_{w}^{\;G}\cap P\rangle=W_G(P),
\end{align*}
and therefore also $(t\hat W_G(\hat P))^{p^a}=1$. Hence we can change notation and assume that $w=\delta _{n+1}$ and $P$ is a Sylow $p$-subgroup of~$G$.
Using also Lemma~\ref{l-w}, we can assume that $G$ has no nontrivial normal $p$-soluble subgroups. We now find ourselves under the hypotheses of Proposition~\ref{4}. By this proposition we obtain normal subgroups $K_1$ and $K_2$ such that the image of $tW(P)$ in $G/K_1$ has exponent dividing $p^{a-1}$ and $W(P)\leq K_2$.

In the group $G/K_2$, a Sylow $p$-subgroup is soluble of derived length $n+1$, and therefore the non-$p$-soluble length of $G/K_2$ is bounded in terms of $n$ by Lemma~\ref{t-exp-sol}.
By Lemma~\ref{l-w} the group $G/K_1$ satisfies the hypothesis of the theorem with a smaller value of exponent. By induction on $a$ we obtain that the non-$p$-soluble length of $G/K_1$ is bounded in terms of $n$ and~$a$. Note that the basis of this induction is the case $a=0$, when $tW(P)\leq K_1$ whence $W(P)\leq K_1$, so that a Sylow $p$-subgroup is soluble of derived length $n+1$, which is covered by Lemma~\ref{t-exp-sol} as above. Since $K_1\cap K_2$ embeds in a direct product of almost simple groups and therefore has non-$p$-soluble length at most~1, the result follows.
\end{proof}

\section{Bounding $p$-length of finite $p$-soluble groups}
\label{s-pl}

In this section we prove Theorem~\ref{solu}.
Recall that a finite group is \emph{$p$-soluble} if it has a normal series each factor of which is either a $p$-group, or a $p'$-group; the minimal number of $p$-factors in such a series is called the \emph{$p$-length} of the group.
We use results on the connection between the $p$-length of a finite $p$-soluble group and its derived length and exponent
first obtained in the seminal Hall--Higman paper \cite{ha-hi} for odd primes and later extended to $p=2$ by others \cite{hoa,ber-gro,bry81}. Moreover, we also use the results in representation theory from these papers, which we state here for convenience.

\begin{theorem}[{\cite[Theorem 2.1.1]{ha-hi}}] \label{t211}
Let $H$ be a $p$-soluble linear group over a field of characteristic
$p$, with no normal $p$-subgroup greater than 1. If $a$ is an element of
order $p^m$ in $H$, then the minimal polynomial of $a$ is $(X-1)^r = 0$, where $r = p^m$,
unless there is an integer $m_0$, not greater than $m$, such that $p^{m_0}-1$ is a power
of a prime $q$, in which case, if $m_0$ is the least such integer, $p^{m-m_0}(p^{m_0}-1)\leq r\leq p^m$.
\end{theorem}

As stated in the Corollary to this theorem in \cite{ha-hi}, if $p$ is neither $2$ nor a Fermat prime, then $r=p^m$; if $p$ is odd, then $r\geq (p-1)p^{m-1}$ always; and if $p=2$, then $r\geq 3\cdot 2^{m-2}$ always.

An element $g$ of order $p^m$ for which the degree of the minimal polynomial is strictly less than $p^m$ is called \emph{exceptional}. Clearly, then $p$ must be either a Fermat prime or~$2$.

\begin{theorem}[{\cite[Theorem 2.1.2]{ha-hi}}] \label{t212}
Under the hypotheses of Theorem~\ref{t211}, if
$a$ and $b$ are exceptional elements of the same Sylow $p$-subgroup of $H$ of orders $p^m$ and $p^n$, respectively, then $[a^{p^{m-1}},b^{p^{n-1}}]=1$.
\end{theorem}

For $p=2$ the original paper \cite{ha-hi} did not give any bounds for the $2$-length of $2$-soluble groups with Sylow $2$-subgroup of exponent $2^e$. First the bound $3e-1$ was obtained by Hoare \cite{hoa}, which was later improved to $2e-1$ by Gross \cite{gro}, and finally to the best-possible $e$ by Bryukhanova \cite{bry79}. However, it is the representation theory technique of Hoare \cite{hoa} that we managed to use in the case of $p=2$ in the proof of Theorem~\ref{solu}.

\begin{theorem}[{\cite[Theorem 2]{hoa}}] \label{t-ho}
Under the hypotheses of Theorem~\ref{t211} for $p=2$, if $a$ and $b$ are elements of the same Sylow $2$-subgroup of $H$ of orders $2^m$ and $2^n$, respectively, such that $[a^{2^{m-1}},b^{2^{n-1}}]\ne 1$, then either
$(a-1)^{2^{m}-1}\ne 0$ or $(b-1)^{2^{n}-1}(a-1)^{2^{m-1}}\ne 0$.
\end{theorem}

We say that a section $L/M$ of a group $G$ \emph{is contained in a section} $N/S$ if $N\geq L\geq M\geq S$, and that $L/M$ \emph{is contained in a normal subgroup} $N$ if $L\leq N$. We say that a subgroup $F$ of $G$ \emph{covers} a section $L/M$ if $(F\cap L)M=L$. The centralizer of a section $L/M$ in a subgroup $H$ is defined as usual: $C_H(L/M)=\{g\in H\mid [L,g]\leq M\}$.

We now prove a key proposition, from which Theorem~\ref{solu} will follow.

\begin{proposition}\label{p-k} Let $K$ be a normal subgroup of a finite $p$-soluble group~$G$. Suppose that $K$ contains a $p$-subgroup $U$ that covers all chief $p$-factors of $G$ contained in $K$ that are not central in $K$ and that for some $t\in G$ the coset $tU$ has exponent dividing~$p^e$. Then the $p$-length of $K$ is at most $e+1$ if $p$ is an odd prime that is not a Fermat prime, at most $2e+1$ if $p$ is a Fermat prime, and at most $3e+1$ if $p=2$.
\end{proposition}

In the proof of Proposition~\ref{p-k}, for odd $p$ we use Theorem~\ref{t211} and \ref{t212} of Hall and Higman \cite{ha-hi} as in the proof of Theorem~3* in Wilson's paper \cite{wil83}.
For $p=2$ we apply similar but more complicated arguments based on Theorem~\ref{t-ho} of Hoare~\cite{hoa}.

\begin{proof}
 We use induction on $e$, where $p^e$ is the exponent of the coset $tU$. The basis of induction is the case $e=0$, which corresponds to $U=1=t$.
 Then all chief $p$-factors of $G$ contained in $K$ are central in $K$ and therefore $K$ has a normal $p$-com\-ple\-ment, so that its $p$-length is at most~$1$. We assume $e>0$ in what follows.
\smallskip

 Let $V=L/M$ be a chief $p$-factor of $G$ contained in $K$ and covered by $U$,
and let $C=C_G(L/M)$. Then $G/C$ acts by conjugation on $V$, which can be regarded as a faithful ${\Bbb F} _p(G/C)$-module. Since this is an irreducible module, $G/C$ has no non-trivial normal $p$-subgroups.

 Let $h\in tU$. If $u\in U$, then $hu\in tU$, and so $h^{p^e}=(hu)^{p^e}=1$, which implies
$$
u^{h^{p^e-1}}\cdots u^{h^{2}}u^{h}u =1.
$$
Since $(U\cap L)M=L$,
this means that the linear transformation of $V$ induced by $h$, which we denote by the same letter, satisfies (in characteristic $p$)
\begin{equation}\label{e-lte}
0=h^{p^e-1}+\cdots +h^{2}+h+1 =(h-1)^{p^e-1}.
\end{equation}
Then either $h^{p^{e-1}}$ acts trivially on $V$, or the linear transformation induced by $h$ is exceptional. It makes sense to split the argument into two cases, for $p$ odd, and for $p=2$.

\medskip

 {\it  Case of $p$ odd.} We define \begin{itemize}
\item[$N\,$] to be the normal closure in $G$ of the subgroup $\langle a^{p^{e-1}}\mid a\in tU\rangle$, and
    \item[$N_1$] to be the normal closure in $G$ of the subgroup $\langle [a^{p^{e-1}},b^{p^{e-1}}]\mid a,b\in tU\rangle $.
    \end{itemize}

\begin{lemma}\label{l-5}
{\rm (a)} The subgroup $N_1$ acts trivially on each chief $p$-factor $L/M$ of $G$
contained in $K$ and covered by~$U$.

{\rm (b)} The subgroup $N$ acts trivially on each chief $p$-factor $L/M$ of $G$ contained in $K$ and covered by $U$
such that
$N_1\cap K\leq M$.

{\rm (c)} The $p$-length of $N\cap K$ is at most~$2$.

{\rm (d)} If $p$ is not a Fermat prime, then $N$ acts trivially on each chief $p$-factor $L/M$ of $G$
contained in $K$ and covered by $U$
and the $p$-length of $N\cap K$ is at most~$1$.
\end{lemma}

\begin{proof}
By Theorem~\ref{t212}, the $p^{e-1}$th powers of exceptional elements of order $p^e$ commute. Therefore by \eqref{e-lte} every commutator $[a^{p^{e-1}},b^{p^{e-1}}]$ for $a,b\in tU $ acts trivially on $V=L/M$, and then so does their normal closure $N_1$. This proves (a).

Now suppose in addition that $N_1\cap K\leq M$. Let $h\in tU$ and $u\in U\cap L$, and let
$$
y=u^{h^{p^{e-1}-1}}\cdots u^{h^{2}}u^{h}u ,
$$
so that $(hu)^{p^{e-1}} = h^{p^{e-1}}y$.
Then $[(hu)^{p^{e-1}},h^{p^{e-1}}]\in N_1$. We also have
$$
[(hu)^{p^{e-1}},h^{p^{e-1}}]=[h^{p^{e-1}}y,h^{p^{e-1}}]=[y,h^{p^{e-1}}]\in K,
$$
since $y\in L\leq K$. As a result, $
[(hu)^{p^{e-1}},h^{p^{e-1}}]=[y,h^{p^{e-1}}]\in N_1\cap K\leq M$, that is,
$$
[u^{h^{p^{e-1}-1}}\cdots u^{h^{2}}u^{h}u,\,h^{p^{e-1}}]\in M.
$$
Since $L=(U\cap L)M$, this implies that the linear transformation of $V$ induced by every element $h\in tU$ satisfies
$$
0=(h^{p^{e-1}-1}+\cdots +h+1)(h^{p^{e-1}}-1) =(h-1)^{2p^{e-1}-1}.
$$
On the other hand, if $h$ induces an automorphism of order $p^e$, then by Theorem~\ref{t211} its minimal polynomial has degree at least $(p-1)p^{e-1}$. Since $(p-1)p^{e-1}>2p^{e-1}-1$ for odd~$p$, it follows that for any $h\in tU$ the elements $h^{p^{e-1}}$ act trivially on~$V$. Then their normal closure $N$ also acts trivially on~$V$, so that (b) holds.

By part (a),
$N_1$ acts trivially on every chief $p$-factor of $G$ contained in $K$ and covered by~$U$. By hypothesis, $K$ acts trivially on all other chief $p$-factors of $G$ contained in~$K$. Therefore all chief $p$-factors of $G$ contained in $N_1\cap K$ are central in $N_1\cap K$, which means that $N_1\cap K$ has a normal $p$-com\-ple\-ment. By part (b),
$N$ acts trivially on every chief $p$-factor of $G$ contained in $(N\cap K)/ (N_1\cap K)$ and covered by~$U$. By hypothesis, $K$ acts trivially on all other chief $p$-factors of $G$ contained in $(N\cap K)/ (N_1\cap K)$. Therefore all chief $p$-factors of $G$ contained in $(N\cap K)/(N_1\cap K)$ are central in $(N\cap K)/(N_1\cap K)$, which means that $(N\cap K)/(N_1\cap K)$ has a normal $p$-com\-ple\-ment. As a result, the $p$-length of $N\cap K\geq N_1\cap K\geq 1$ is at most~2.

If $p$ is an odd prime that is not a Fermat prime, then no exceptional elements appear, so that all elements $a^{p^{e-1}}$ for $a\in tU$ act trivially on~$V$, and hence so does~$N$. Then all chief $p$-factors of $G$ contained in $N\cap K$ are central in $N\cap K$, which implies that $N\cap K$ has a normal $p$-com\-ple\-ment and $p$-length~$1$.
\end{proof}

We now complete the proof of Proposition~\ref{p-k} for odd~$p$. We define $N$ as in Lemma~\ref{l-5} and consider $\bar G=G/N$ denoting by the bar images of elements and subgroups in $\bar G$.
Given a chief $p$-factor $L/M$ of $G$,   either $L\leq N$, so that $L/M$ is  trivial in $\bar G$, or $M\geq N$, in which case $C_{\bar G}(\bar L/\bar M)=\overline{C_G(L/M)}$. Therefore $\bar U$ covers all chief $p$-factors of $G$ contained in $\bar K$ that are not central in $\bar K$. Thus,  the hypotheses
of Proposition~\ref{p-k} hold for the images with $\bar t \bar U$ of exponent dividing $p^{e-1}$ by the construction of~$N$. By induction, the $p$-length of $\bar K\cong K/(K\cap N)$ is at most $2(e-1)+1$ (or at most $e$ if $p$ is not a Fermat prime). By Lemma~\ref{l-5} the $p$-length of $K\cap N$ is at most 2 (respectively, 1), and the result follows.
\medskip

 {\it  Case $p=2$.} Recall that we assume $e\geq 1$. We need one more step in the basis of induction, as some of the arguments below do not work for $e=1$. Thus, suppose that $e=1$. Let $V=L/M$ be any chief 2-factor of $G$ contained in $K$ and covered by~$U$. Every element $\bar x\in V$ is the image of an element $x\in U\cap L$. For any $u\in U$ we have $1=(tu)^2=(tux)^2=(tu)^2x[x,tu]x$, and since $x^2\in M$ this implies $[x,tu]\in M$. Thus, $tU\leq C_G(V)$. Hence the normal closure $B$ of $tU$ in $G$ has the property that all chief 2-factors of $G$ contained in $B\cap K$ are central in $B\cap K$, whence $B\cap K$ has a normal 2-com\-ple\-ment and 2-length 1. All chief 2-factors of $G$ contained in $K/(B\cap K)$ are not covered by $U\leq B\cap K$ and therefore are central in $K$ by hypothesis. Therefore $K/(B\cap K)$ also has a normal 2-com\-ple\-ment and 2-length 1. As a result, the 2-length of $K$ is at most 2, as required.

From now on we assume that $e\geq 2$. We define
\begin{itemize}
 \item[$T\,$] to be the normal closure in $G$ of the subgroup $\langle a^{2^{e-1}}\mid a\in tU\rangle$,
 \item[$T_1$] to be the normal closure in $G$ of the subgroup $\langle [a^{2^{e-1}},\,b^{2^{e-2}}]\mid a,b\in tU \rangle $, and
 \item[$T_2$] to be the normal closure in $G$ of the subgroup $\langle [a^{2^{e-1}},\,b^{2^{e-1}}]\mid a,b\in tU \rangle $.
 \end{itemize}
 Note that $T\geq T_1\geq T_2$.

\begin{lemma}\label{l-52}
{\rm (a)} The subgroup $T_2$ acts trivially on each chief $2$-factor $L/M$ of $G$ contained in $K$ and covered by~$U$.

{\rm (b)} The subgroup $T_1$ acts trivially on each chief $2$-factor $L/M$ of $G$
contained in $K$ and covered by $U$ such that
$T_2\cap K\leq M$.

{\rm (c)} The subgroup $T$ acts trivially on each chief $2$-factor $L/M$ of $G$ contained in $K$ and covered by $U$ such that
$T_1\cap K\leq M$.

{\rm (d)} The $2$-length of $T\cap K$ is at most~$3$.
\end{lemma}

\begin{proof} Let $V=L/M$.  By \eqref{e-lte} for any $h\in tU$ and $u\in U$ we have
$(h-1)^{2^{e}-1}=1$ and $ ((hu)-1)^{2^{e}-1}=0$,
and the more so,
$$
0=((hu)-1)^{2^{e}-1} (h-1)^{2^{e-1}}.
$$
By Theorem~\ref{t-ho} applied with $a=h$, $b=hu$, and $n=m=e$, then $[h^{2^{e-1}},(hu)^{2^{e-1}}]$ acts trivially on~$V$. Therefore so does the subgroup $T_2$ and part (a) is proved.

Now suppose in addition that $T_2\cap K\leq M$. The commutator
$[(hux)^{2^{e-1}}, h^{2^{e-1}}]$ obviously belongs to~$T_2$. This commutator also belongs to $K$, since
$(hux)^{2^{e-1}}=h^{2^{e-1}}y$, where $y\in K$, since $u,x\in U\leq K$, and then
$$
[(hux)^{2^{e-1}}, h^{2^{e-1}}]=[h^{2^{e-1}}y, h^{2^{e-1}}]=[y, h^{2^{e-1}}]\in K.
$$
Thus,
$[(hux)^{2^{e-1}}, h^{2^{e-1}}]\in T_2\cap K\leq M$ by our assumption. For the same reasons,
$[(hu)^{2^{e-1}}, h^{2^{e-1}}]\in M$. As a result, modulo $M$ we have the congruences
\begin{align*}
1 &\equiv [(hux)^{2^{e-1}},\, h^{2^{e-1}}]\\&=[(hu)^{2^{e-1}}x^{(hu)^{2^{e-1}-1}}\cdots x^{hu}x,\, h^{2^{e-1}}]\\
                                               &\equiv [x^{(hu)^{2^{e-1}-1}}\cdots x^{hu}x, \,h^{2^{e-1}}]  .
\end{align*}
In terms of linear transformations this implies that
$$
0=((hu)^{2^{e-1}-1}+\cdots + hu +1)(h^{2^{e-1}}-1)= ((hu)-1)^{2^{e-1}-1} (h-1)^{2^{e-1}}.
$$
We also have $
(h-1)^{2^{e}-1}=0
$ by \eqref{e-lte}.
We now apply Theorem~\ref{t-ho} with $a= h$, $b=hu$, $m=e$, and $n=e-1$:
then $[h^{2^{e-1}},(hu)^{2^{e-2}}]$ acts trivially on~$V$. Therefore so does the subgroup $T_1$ and part (b) is proved.

Finally, suppose in addition that $T_1\cap K\leq M$. We have $[(hx)^{2^{e-1}}, h^{2^{e-2}}]\in T_1$. We also have $(hx)^{2^{e-1}}= h^{2^{e-1}}z$ for $z\in K$, so that $$
[(hx)^{2^{e-1}}, h^{2^{e-2}}]=[h^{2^{e-1}}z, h^{2^{e-2}}]=[z,h^{2^{e-2}}]\in K.
$$
 Hence, $[(hx)^{2^{e-1}}, h^{2^{e-2}}]\in T_1\cap K\leq M$.
As a result, modulo $M$ we have the congruences
\begin{align*}
1 &\equiv [(hx)^{2^{e-1}},\, h^{2^{e-2}}]\\&=[h^{2^{e-1}}x^{h^{2^{e-1}-1}}\cdots x^{h}x,\, h^{2^{e-2}}]\\
                                               &\equiv [x^{h^{2^{e-1}-1}}\cdots x^{h}x, \,h^{2^{e-2}}]  .
\end{align*}
In terms of linear transformations this implies that
\begin{align*}
0&=(h^{2^{e-1}-1}+\cdots + h +1)(h-1)^{2^{e-2}}\\
&= (h-1)^{2^{e-1}-1} (h-1)^{2^{e-2}} = (h-1)^{3\cdot 2^{e-2}-1}.
\end{align*}
 On the other hand, if $h$ induces an automorphism of order $2^e$, then by Theorem~\ref{t211} its minimal polynomial has degree at least $3\cdot 2^{e-2}$. Hence it follows that $h^{2^{e-1}}$ acts trivially on $V$, for every $h\in tU$. Therefore so does $T$, and part (c) is proved.

We set for convenience $T_0=T$ and $T_3=1$. By parts (a)--(c), for every $i\in \{0,1,2\}$ the subgroup $T_i$ acts trivially on every chief $2$-factor of $G$ contained in $K/(T_{i+1}\cap K)$ and covered by~$U$. By hypothesis, $K$ acts trivially on all other chief $2$-factors of $G$ contained in $K/(T_{i+1}\cap K)$. Therefore all chief $2$-factors of
$G$ contained in $(T_i\cap K)/(T_{i+1}\cap K)$ are central in $(T_i\cap K)/(T_{i+1}\cap K)$, which means that $(T_i\cap K)/(T_{i+1}\cap K)$ has a normal $2$-com\-ple\-ment.
 As a result, the $2$-length of
$$T\cap K\geq T_1\cap K\geq T_2\cap K\geq 1$$
is at most 3.
\end{proof}

We now complete the proof of Proposition~\ref{p-k} for $p=2$. We define $T$ as in Lemma~\ref{l-52}.
Consider $\bar G=G/T$ denoting also by the bar all images in~$\bar G$.
Given a chief $2$-factor $L/M$ of $G$,   either $L\leq T$, so that $L/M$ is  trivial in $\bar G$, or $M\geq T$, in which case $C_{\bar G}(\bar L/\bar M)=\overline{C_G(L/M)}$.
 Therefore $\bar U$ covers all chief $2$-factors of $G$ contained in $\bar K$ that are not central in~$\bar K$. Thus,  the hypotheses
of Proposition~\ref{p-k} hold with $\bar t \bar U$ of exponent dividing $2^{e-1}$ by the construction of~$T$. By induction, the $2$-length of $\bar K\cong K/(K\cap T)$ is at most $3(e-1)+1$. By Lemma~\ref{l-52} the $2$-length of $K\cap T$ is at most 3, and the result follows.
\end{proof}

\begin{proof}[Proof of Theorem~\ref{solu}] Recall that $w$ is a multilinear commutator word of weight $n$, while $P_w$ is the set of $w$-values on elements of a subgroup $P$ of a group $G$ and $W_G(P)=\langle P_w^{\;G}\cap P\rangle$. We have a normal $p$-soluble subgroup $H$ of a finite group $G$, and a Sylow $p$-subgroup $P$ of $H$ such that for some $t\in G$ the coset $tW_G(P)$ is of exponent dividing $p^a$. Our aim is bounding the $p$-length of $H$ in terms of $a$ and $n$ only.
We can obviously assume that $G=H\langle t\rangle$ (by Lemma~\ref{l-w}(a)); in particular, then $G$ is $p$-soluble, which we assume in what follows.

\begin{lemma}\label{l-chief}
If $L/M$ is a chief $p$-factor of $G$ contained in $H$ and $H/C_H(L/M)$ has $p$-length at least $n+1$, then $W_G(P)$ covers $L/M$, that is, $(W_G(P)\cap L)M=L$.
\end{lemma}

\begin{proof}
By the theorems of Hall and Higman \cite[Theorem~A]{ha-hi} for $p\ne 2$ and of
Bryukhanova~\cite{bry81} for $p=2$, the image of $P^{(n)}$ in $H/C_H(L/M)$ is nontrivial. Since $w(P)\geq P^{(n)}$ by Lemma~\ref{delta}, it follows that the image of $w(P)$ in $H/M$ does not centralize $L/M$. Since $L/M$ is a normal subgroup of the image of $P$ in $G/M$, by Lemma~\ref{lbw} there is a value $v$ of the word $w$ on elements of $P$ such that the image of $v$ is nontrivial in  $L/M$. Since $L/M$ is a chief factor of $G$, there are conjugates $v^{g_{i}}$ for some $g_i\in G$ whose images generate $L/M$. There are also elements $l_i\in L$ such that $v^{g_{i}l_i}\in P\cap L$, since $P\cap L$ is a Sylow $p$-subgroup of~$L$. By definition, then $v^{g_{i}l_i}\in W_G(P)$, and since $v^{g_{i}l_i}M=v^{g_{i}}M$, the result follows.
\end{proof}

We now finish the proof of Theorem~\ref{solu}.
Let ${\mathscr D}$ be the set of all chief $p$-factors $V$ of $G$ such that $H/C_H(V)$ has $p$-length at most~$n$. We define the subgroup
$$
D=\bigcap _{V\in {\mathscr D}} C_H(V).
$$
Then $H/D$ also has $p$-length at most~$n$. It remains to prove that the $p$-length of $D$ is bounded. This follows from Proposition~\ref{p-k} applied with $K=D$ and $U=W_G(P)\cap D$. The required properties hold by Lemma~\ref{l-chief} and because every chief $p$-factor of $G$ contained in $D$ that belongs to ${\mathscr D}$ is central in $D$ by construction.
\end{proof}

\begin{proof}[Proof of Corollary~\ref{c-exp-nil}] We have a $p$-soluble group $G$ whose Sylow $p$-subgroup $P$ is such that
$(\gamma _{c+1}(P))^{p^a}=1$; our aim is a bound for the $p$-length of $G$ in terms of $c$ and~$a$. If we set $w=[x_1,x_2,\dots,x_{c+1}]$, then $w(P)=\gamma _{c+1}(P)$. The subgroup $W_G(P)\leq P$ is generated by conjugates of $w$-values on elements of~$P$. By hypothesis these conjugates are of orders dividing $p^a$. Since $P/\gamma _{c+1}(P)$ is nilpotent of class $c$, the group $W_G(P)/\gamma _{c+1}(P)$ has exponent dividing $p^{ac}$. Therefore the group $W_G(P)$ has exponent dividing $p^{ac+a}$. The result now follows from Theorem~\ref{solu} applied with $H=G$ and $t=1$.
 \end{proof}

\section{Length results for profinite groups}
\label{s-lpro}

In this section we derive Theorems~\ref{profi} and~\ref{pro-p-solu} from their finite analogues proved in \S\S\,\ref{s-nsl} and~\ref{s-pl}. This is done by standard arguments; we find it convenient to refer to some lemmas in Wilson's paper \cite{wil83}.

\begin{lemma}[{\cite[Lemma~2]{wil83}}]\label{l-wil2}
Let $\frak X_1,\dots ,\frak X_m$ be classes of finite groups closed with respect to normal subgroups and subdirect products, and let $\frak X$ be the class of groups $X$ having a series
$$
1=X_0\leq X_1\leq \dots\leq X_m=X
$$
with $X_i/X_{i-1}\in \frak X_i$ for each~$i$. Then every pro-$\frak X$ group $G$ has a series
$$
1=G_0\leq G_1\leq \dots\leq G_m=G
$$
of closed characteristic subgroups such that $G_i/G_{i-1}$ is a pro-$\frak X_i$ group for each~$i$.
\end{lemma}

\begin{lemma}[{\cite[Lemma~3]{wil83}}]\label{l-wil3}
Let $\frak Y$ be a class of groups consisting of non-abelian finite simple groups, and let $\frak X$ be the class of finite direct products of $\frak Y$-groups. Then a profinite group $G$ is a pro-$\frak X$ group if and only if it is isomorphic (as a topological group) to a Cartesian product of $\frak Y$-groups.
\end{lemma}

Recall that $w$ is a multilinear commutator word of weight $n$,
and $W_G(P)$ is the closed subgroup generated by all elements of $P$ that are conjugate in $G$ to $w$-values on elements of~$P$. We often write $W(P)$ in place of $W_G(P)$. The next lemma is similar to \cite[Lemma~1]{wil83}

\begin{lemma}\label{l-coset}
Let $G$ be a profinite group and suppose that for some prime $p$ the subgroup $W(P)$ of a $p$-Sylow subgroup $P$ of $G$ is a periodic group. Then there is an open normal subgroup $H$ of $G$ and an element $t\in W(P)$ such that the coset $t(H\cap W(P))$ has finite exponent.
\end{lemma}

\begin{proof} Write $S_e =\{g\in W(P)\mid g^e =1 \}$ for each integer $e\geq 1$. Each set $S_e$ is closed in the subgroup topology and $W(P) =\bigcup _{e=1}^{\infty} S_e$. It follows from Baire's category theorem (see for example \cite[p. 200]{ke}) that some set $S_e$ has non-empty interior and therefore contains a coset $tS$ for some open subgroup $S$ of $W(P)$. Since $S$ is open, there is an open subgroup $H$ of $G$ such that $H\cap W(P)\leq S$; and clearly $H$ can be chosen to be normal in~$G$. The result follows.
\end{proof}

\begin{proof}[Proof of Theorem~\ref{profi}] We have a profinite group $G$ such that $W(P)$ is locally finite for a $p$-Sylow subgroup $P$ of $G$ for some $p\in\pi(G)$. We need to show that $G$ has a finite characteristic series in which each factor either is pro-$p$-soluble or is isomorphic to a Cartesian product of
non-abelian finite simple groups of order divisible by~$p$. By Lemma~\ref{l-coset}
there is an open normal subgroup $H$ of $G$ and an element $t\in W(P)$ such that the coset $t(H\cap W(P))$ has finite exponent, which is of course a $p$-power, say, $p^a$. Consider any finite continuous homomorphic image $\bar G$ of $G$ denoting by bars the images of elements or subgroups. Then $\overline{P\cap H}$ is a Sylow $p$-subgroup of $\bar H$ and $\bar t W_{\bar G}(\overline{P\cap H})\subseteq \overline{tW_G(P)}$ by Lemma~\ref{l-w}, so that also $(\bar t W_{\bar G}(\overline{P\cap H}))^{p^a}=1$. Therefore $\bar G$ satisfies the hypotheses of Theorem~\ref{gera}, by which the non-$p$-soluble length of $\bar H$ is bounded in terms of $n$ (weight of $w$) and~$a$. Hence $\bar G$ has a series
$$
1=H_0\leq H_1\leq \dots\leq H_m=\bar G
$$
of length $m$ bounded in terms of $a$, $n$, and $|G/H|$ such that
 $H_i/H_{i-1}$ is $p$-soluble for $i$ odd and is a direct product of non-abelian finite simple groups of order divisible by $p$ for $i$ even. By Lemma~\ref{l-wil2} we obtain that $G$ has a series of closed characteristic subgroups
$$
1=G_0\leq G_1\leq \dots\leq G_m= G
$$
of the same finite length $m$ such that
 $G_i/G_{i-1}$ is pro-$p$-soluble for $i$ odd and is a projective limit of direct products of non-abelian finite simple groups of order divisible by $p$ for $i$ even. The factors of the second type are isomorphic to Cartesian products of non-abelian finite simple groups of order divisible by $p$ by Lemma~\ref{l-wil3}.
\end{proof}

\begin{proof}[Proof of Theorem~\ref{pro-p-solu}]
We have a pro-$p$-soluble group $G$ such that $W(P)$ is locally finite for a $p$-Sylow subgroup $P$ of~$G$. We need to show that $G$ has a finite series of closed characteristic subgroups with factors either pro-$p$ groups or pro-$p'$ groups.
By Lemma~\ref{l-coset}
there is an open normal subgroup $H$ of $G$ and an element $t\in W(P)$ such that the coset $t(H\cap W(P))$ has finite $p$-power exponent $p^a$, say. Just as above in the proof of Theorem~\ref{profi}, in any finite continuous homomorphic image $\bar G$ of $G$, the subgroup $\overline{P\cap H}$ is a Sylow $p$-subgroup of $\bar H$ and $\bar t W_{\bar G}(\overline{P\cap H})\subseteq \overline{tW_G(P)}$ by Lemma~\ref{l-w}, so that also $(\bar t W_{\bar G}(\overline{P\cap H}))^{p^a}=1$. Therefore $\bar G$ satisfies the hypotheses of Theorem~\ref{solu}, by which $\bar H$ has $p$-length bounded in terms of $n$ and~$a$. Hence $\bar G$ has a series
$$
1=H_0\leq H_1\leq \dots\leq H_m=\bar G
$$
of length $m$ bounded in terms of $a$, $n$, and $|G/H|$ such that
 $H_i/H_{i-1}$ is a $p'$-group for $i$ odd, and a $p$-group for $i$ even. By Lemma~\ref{l-wil2} we obtain that $G$ has a series of closed characteristic subgroups
$$
1=G_0\leq G_1\leq \dots\leq G_m= G
$$
of the same finite length $m$ such that
 $G_i/G_{i-1}$ is pro-$p'$ group for $i$ odd, and a pro-$p$ group for $i$ even.
\end{proof}

\begin{proof}[Proof of Corollary~\ref{prosolu}]
We have a prosoluble group $G$ such that $\pi(G)$ is finite and $W(P)$ is locally finite for any $p$-Sylow subgroup $P$ of $G$, for any $p\in\pi(G)$. We need to show that $ph(G)$ is finite, that is, $G$ has a finite characteristic series with pronilpotent factors.
Fix a prime $p\in \pi (G)$. By Theorem~\ref{pro-p-solu} the group
 $G$ has a finite series of closed characteristic subgroups
in which each factor is either a pro-$p$ group, or a pro-$p'$ group. Each pro-$p'$ factor in this series has a smaller set of prime divisors, so the result follows by an easy induction on the cardinality of $\pi (G)$, since the hypotheses are inherited by subgroups and homomorphic images by Lemma~\ref{l-w}.
\end{proof}

\section{Main results}

In what follows we use without special references Zelmanov's theorem \cite{zel92} on local finiteness of periodic profinite groups. In particular, when proving local finiteness of a subgroup, we can freely factor out any periodic closed normal  subgroup.

Recall that $w$ is a multilinear commutator word of weight $n$, the set of $w$-values on a subgroup $P\leq G$ is denoted by $P_w$, the corresponding closed verbal subgroup is $w(P)$, and $W_G(P)$ (often denoted by $W(P)$) is the closed subgroup generated by all elements of $P$ that are conjugate in $G$ to elements of $P_w$. For brevity we denote by $\mathscr{X}$ the class of all profinite groups $G$ in which all $w$-values have finite order and $w(P)$ is periodic for any $p$-Sylow subgroup of $G$ (for any prime $p\in \pi (G)$). We denote by $\mathscr{Y}$ the class of all profinite groups $G$ in which all $w$-values have finite order and $W(P)$ is periodic for any $p$-Sylow subgroup of $G$ (for any prime $p\in \pi (G)$). Clearly, $\mathscr{Y}\subseteq \mathscr{X}$. Both classes $\mathscr{X}$ and $\mathscr{Y}$ are closed with respect to subgroups and homomorphic images, which is obvious for $\mathscr{X}$, and follows from Lemma~\ref{l-w} for $\mathscr{Y}$.
This is the main reason why we found it easier to work with these classes rather than the smaller class of groups satisfying the hypotheses of the main Theorem~\ref{main}.

The (virtually) soluble case of Theorem~\ref{main} is covered by the following theorem in \cite{dms3}, which is also used throughout what follows.

\begin{theorem}[{\cite[Theorem~3]{dms3}}]\label{dems}
Let $w$ be a multilinear commutator word. Let $G$ be a virtually soluble profinite group in which all $w$-values have finite order. Then $w(G)$ is locally finite and has finite exponent.\end{theorem}

We approach the proof of the main result in a series of lemmas. 

\begin{lemma}\label{sta00}
Let $G\in\mathscr{X}$ be a group having a closed
normal abelian subgroup $M$ such that $G/M$ is locally finite. Then $G$ possesses a series of closed normal subgroups $1\leq T\leq C\leq G$ such that $T$ has finite exponent, $[M,C]\leq T$, and $G/C$ is virtually soluble.
\end{lemma}

\begin{proof} First consider the case where $w=\delta _k$.
Choose arbitrarily $g_1,\dots,g_{2^k}\in G$. The subgroup $M\langle g_1,\dots,g_{2^k}\rangle$ is virtually abelian and so, by Theorem~\ref{dems}, the verbal subgroup $w(M\langle g_1,\dots,g_{2^k}\rangle)$ is locally finite and has finite exponent. Denote by $e(g_1,\dots,g_{2^k})$ the exponent of $M\cap w(M\langle g_1,\dots,g_{2^k}\rangle)$. For each positive integer $i$ we set
$$
S_i=\{(g_1,\dots,g_{2^k})\in G\times\dots\times G \mid e(g_1,\dots,g_{2^k})=i\}.
$$
The sets $S_i$ are closed in $G\times\dots\times G$ and they cover the whole group $G\times\dots\times G$. By Baire's category theorem there exists $i$ such that $S_i$ contains an open subset of $G\times\dots\times G$. Therefore there exist elements $a_1,\dots,a_{2^k}\in G$, an open subgroup $H\leq G$, and a positive integer $i$ such that $e(a_1h_1,\dots,a_{2^k}h_{2^k})=i$ for any $h_1,\dots,h_{2^k}\in H$. Let $T$ be the subgroup generated by all elements of $M$ which have finite order dividing~$i$. Obviously $T$ is a closed normal subgroup of~$G$. Let $C$ be the full preimage of $C_{G/T}(M/T)$, which is also a closed normal subgroup of~$G$. Since $M$ is abelian, it follows that $T$ has finite exponent dividing~$i$. Consider the quotient $G/T$. Thus, we assume that $i=1$. In particular, we deduce that $[M,w(a_1h_1,\dots,a_{2^k}h_{2^k})]=1$ for any $h_1,\dots,h_{2^k}\in H$. In other words, $w(a_1h_1,\dots,a_{2^k}h_{2^k})\in C$ for any $h_1,\dots,h_{2^k}\in H$. Consider the quotient $G/C$. In the quotient we have $w(a_1h_1,\dots,a_{2^k}h_{2^k})=1$ for any $h_1,\dots,h_{2^k}\in H$. By \cite[Lemma 2.2]{mz} the image of $H$ in $G/C$ is soluble. The proof is complete in the special case of $w=\delta _k$. The general case follows, since every $\delta _n$-value is a $w$-value by Lemma~\ref{delta}, so that our group in the class $\mathscr{X}$ corresponding to $w$ is also contained in the same kind of class corresponding to~$\delta _n$.
\end{proof}

\begin{lemma}\label{sta1}
Let $G\in\mathscr{X}$ be a group having a closed normal abelian subgroup $M$ such that $G/M$ is either a locally finite pro-$p$ group or a Cartesian product of isomorphic finite simple groups. Then $w(G)$ is locally finite.
\end{lemma}

\begin{proof} In the case where $w$ has weight~1 the lemma is obvious, so we assume that $w$ is of weight at least~2. Let $T$ and $C$ be as in Lemma~\ref{sta00}. Passing to the quotient $G/T$ we can assume that $T=1$.

Suppose first that $G/M$ is a locally finite pro-$p$ group. Then $C$ is pronilpotent and $w(C)=w(P)$, where $P$ is the $p$-Sylow subgroup of~$C$. Hence, by the hypothesis, $w(C)$ is locally finite and therefore we can pass to the quotient $G/w(C)$. Thus, without loss of generality we assume that $w(C)=1$. It follows that $C$ is soluble. Taking into account that $G/C$ is virtually soluble, we deduce that so is~$G$. Now the lemma is immediate from Theorem~\ref{dems}.

We now assume that $G/M$ is isomorphic to a Cartesian product of isomorphic finite simple groups. Then $C/Z(C)$ has finite exponent and a result of Mann~\cite{mann} tells us that the derived group of $C$ has finite exponent (and so is locally finite). The quotient $G/C'$ is virtually soluble, and so the result again follows from Theorem~\ref{dems}.
\end{proof}

The result of Mann~\cite{mann} we just mentioned can be viewed as a generalization of Schur's theorem on groups with finite quotient by the centre. It is an open problem whether the derived subgroup of a profinite group must be locally finite if the quotient of the group by the centre is locally finite. But such a result is true in our special situation.

\begin{lemma}\label{sta3}
Let $G\in\mathscr{X}$ be a group such that $G/Z(G)$ is locally finite. Then $G'$ is locally finite.
\end{lemma}

\begin{proof}
In view of the result of Wilson \cite[Theorem~1]{wil83},
the group $G$ has a finite series of closed normal subgroups
$$
1\leq Z(G)=G_0\leq G_1\leq\dots\leq G_s=G
$$
each of whose factors $G_{i+1}/G_i$ is either a Cartesian product of isomorphic finite simple groups or a locally finite pro-$p$ group.
We
use induction on $s$, with obvious basis when $s=0$. By induction the derived group $G_{s-1}'$ is locally finite. We can pass to the quotient $G/G_{s-1}'$ and assume that $G_{s-1}$ is abelian. Now Lemma~\ref{sta1} shows that $w(G)$ is locally finite. We pass to the quotient $G/w(G)$ and assume that $G$ is soluble. Since soluble locally finite profinite  groups have finite exponent, we deduce that $G/Z(G)$ has finite exponent. In view of Mann's result \cite{mann} the lemma follows.
\end{proof}

\begin{lemma}\label{sta4}
Let $G\in\mathscr{X}$ be a group having a closed normal abelian subgroup $M$ such that $G/M$ is locally finite. Then $w(G)$ is locally finite.
\end{lemma}

\begin{proof} Let $T$ and $C$ be as in Lemma~\ref{sta00}. Without loss of generality we can assume that $T=1$. By Lemma~\ref{sta3} the derived group of $C$ is locally finite. Passing to the quotient $G/C'$ we can assume that $C$ is abelian. Thus, $G$ is virtually soluble and so the result follows from Theorem~\ref{dems}.
\end{proof}

\begin{lemma}\label{sta5}
Let $G\in\mathscr{X}$ be a group having a closed normal soluble subgroup $M$ such that $G/M$ is locally finite. Then $w(G)$ is locally finite.
\end{lemma}

\begin{proof}
We use induction on the derived length of~$M$. By Lemma~\ref{sta4} the lemma holds if $M$ is abelian. So we assume that $M$ is non-abelian and let $N$ be the last nontrivial term of the derived series of~$M$. By induction the image of $w(G)$ in $G/N$ is locally finite. Again we use Lemma~\ref{sta4} and conclude that $w(w(G))$ is locally finite. We can now pass to the quotient $G/w(w(G))$ and assume that $w(w(G))=1$. In this case $G$ is soluble and the result follows from Theorem~\ref{dems}.
\end{proof}

Recall that for a prosoluble group $G$ we denote by $ph(G)$ the length of a shortest series of (closed) characteristic subgroups all of whose factors are pronilpotent, and if $G$ has no finite series of this kind, then we write $ph(G)=\infty$.

\begin{lemma}\label{sta6}
Let $G\in\mathscr{X}$ be a prosoluble group such that $ph(G)$ is finite. Then $w(G)$ is locally finite.
\end{lemma}

\begin{proof} Suppose first that $G$ is pronilpotent. Then $G=\prod P_i$ is the Cartesian product of its $p$-Sylow subgroups and $w(G)$ is the Cartesian product of the subgroups $w(P_i)$. Since each $w$-value in $G$ has finite order, it follows that $w(P_i)=1$ for all but finitely many Sylow subgroups~$P_i$. Therefore $w(G)$ is the Cartesian product of finitely many locally finite subgroups $w(P_i)$. It follows that $w(G)$ is locally finite.

Hence we can assume that $ph(G)\geq2$ and use induction on $ph(G)$. Let
$$
1=G_0<G_1<\dots<G_{h}=G
$$
be a series of closed normal subgroups of length $h=ph(G)$ with pronilpotent factors. By induction the image of $w(G)$ in $G/G_1$ is locally finite. Moreover, since $G_1$ is pronilpotent, $w(G_1)$ is locally finite. We pass to the quotient $G/w(G_1)$ and assume that $w(G_1)=1$. It follows that $G_1$ is soluble and $w(G)$ is soluble-by-(locally finite). Lemma~\ref{sta5} now tells us that $w(w(G))$ is locally finite. We pass to the quotient $G/w(w(G))$ and assume that $w(w(G))=1$. In this case $G$ is soluble and the result follows from Theorem~\ref{dems}.
\end{proof}

\begin{lemma}\label{sta7}
Let $G$ be a prosoluble group such that $w(G)\neq1$. Then there exist a finite set of primes $\pi$ and a pro-$\pi$ subgroup $H\leq G$ such that $w(H)\neq1$.
\end{lemma}

\begin{proof} Since $w(G)\neq1$, there exists a normal open subgroup $N$ in $G$ such that $w(G/N)\neq1$. Set $\pi=\pi(G/N)$ and choose a $\pi$-Hall subgroup $H$ in~$G$. Then $G=NH$ and therefore $w(H)\neq1$.
\end{proof}

Given two words $u=u(x_1,\dots,x_m)$ and $v=v(x_1,\dots,x_n)$, we denote by $u\circ v$ the word $u(v(x_{11},\dots,x_{1n}),\dots,v(x_{m1},\dots,x_{mn}))$. It is clear that if $u$ and $v$ are multilinear commutators, then so is $u\circ v$. Recall that $\mathscr{Y}$ is the class of all profinite groups $G$ in which all $w$-values have finite order and $W(P)$ is periodic for any $p$-Sylow subgroup of $G$ (for any prime $p\in \pi (G)$).

\begin{lemma}\label{sta8}
Let $G\in\mathscr{Y}$ be a non-soluble prosoluble group. Then $G$ contains a finite subgroup $A$ such that $w(A)\neq1$.
\end{lemma}

\begin{proof} Set $v=w\circ w$. Then $v$ is a multilinear commutator word with the properties that every $v$-value in $G$ has finite order and $v(P)$ is locally finite for every $p$-Sylow subgroup $P\leq G$. According to Lemma~\ref{sta7} there exist a finite set of primes $\pi$ and a pro-$\pi$ subgroup $H$ in $G$ such that $v(H)\neq1$. By Theorem~\ref{prosolu}, then  $ph(H)$ is finite, since $G\in \mathscr{Y}$. Now Lemma~\ref{sta6} tells us that $w(H)$ is locally finite. It is clear that $v(H)\leq w(w(H))$ and therefore $w(w(H))\neq1$. Thus, we can choose a finite subgroup $A\leq w(H)$ such that $w(A)\neq1$.
\end{proof}

Let $G$ be a finite group, on which a finite group $A$ of coprime order acts by automorphisms. It is well-known that if $N$ is a normal $A$-invariant subgroup of $G$, then $C_{G/N}(A)=C_G(A)N/N$ (see for example \cite[Theorem~6.2.2]{go}). (Henceforth we use the centralizer notation for fixed-point subgroups.) A~profinite version of this result can be stated as follows.

\begin{lemma}\label{l-coprime}
Let $G$ be a profinite group on which a finite group $A$ acts by continuous automorphisms. Suppose that $\pi(G)\cap\pi(A)=\varnothing$. If $N$ is a closed normal $A$-invariant subgroup of $G$, then $C_{G/N}(A)=C_G(A)N/N$.
\end{lemma}

A proof of this fact can be found in \cite{io98} (see also Lemma 2 in \cite{he} for the case where $|A|$ is a prime).

In his seminal work \cite{th} Thompson proved that if $G$ is a finite soluble group on which a finite soluble group $A$ of coprime order acts by automorphisms, then the Fitting height $h(G)$ is bounded in terms of $h(C_G(A))$ and the number of prime divisors of $|A|$ counting multiplicities. Further results in this direction were devoted to improving the corresponding bounds, with best possible one obtained by Turull, see his survey \cite{turull}. We need a profinite version of Thompson's theorem for the profinite analogue, which can be deduced by standard
arguments in the spirit of Lemma~\ref{l-wil2}. The fact that the hypotheses are inherited by quotients by closed normal subgroups follows from Lemma~\ref{l-coprime}.
 Therefore we omit the proof.

\begin{proposition}\label{thompson}
Let $G$ be a prosoluble group on which a finite soluble group $A$ acts by continuous automorphisms and suppose that $\pi(G)\cap \pi(A)=\varnothing$. If $ph(C_G(A))$ is finite, then $ph(G)$ is also finite.
\end{proposition}

The following proposition plays a key role in the proof of the main theorem.

\begin{proposition}\label{pi-prosol}
Let $G\in\mathscr{Y}$ be a prosoluble group. Then $ph(G)$ is finite.
\end{proposition}

\begin{proof} Suppose that $ph(G)$ is infinite; then we will arrive at a contradiction. For that we inductively construct
an increasing chain of normal closed subgroups $1=R_1 \le R_2 \le \cdots $,
a decreasing chain of normal open subgroups
$G = H_1 \ge H_2 \ge \cdots $, and subgroups $A_i\geq R_i$, $i=1,2,\dots $,
such that
\begin{enumerate}

\item there exist mutually disjoint sets of primes $\pi_i$ with the property that
 $A_i/R_i$
 is a finite $\pi_i$-group and $w(A_i/R_i)\ne 1$;

\item $[A_i,A_j]\le R_i$ whenever $j<i$;

\item $ph(G/R_i)=\infty$ for every $i$;

\item $
A_i\cap H_{i+1}=R_i$ for every~$i$.

\end{enumerate}

Since $ph(G)$ is infinite, Lemma~\ref{sta8} implies that $G$ contains a finite subgroup $A_1$ such that $w(A_1)\neq1$. Let $\pi_1=\pi(A_1)$ and set $R_1=1$. Next, choose a normal open subgroup $H_2$ in $G$ such that $A_1\cap H_2=1$. This completes the first step of our construction; clearly, then properties 1--4 hold for $i=1$.

We perform the second step for clarity, before the formal inductive construction.
Since the set $\pi_1$ is finite, it follows from Theorem~\ref{pro-p-solu}
that $H_2$ has a finite series
$$
H_2=N_0\geq N_1\geq \dots \geq N_l=1
$$
of closed characteristic subgroups $N_{i}$ in which each quotient $ N_{i}/ N_{i+1}$ is either a pro-$\pi_1$ or a pro-$\pi_1'$ group. By Corollary~\ref{prosolu}, $ph(N_{i}/ N_{i+1})$ is finite whenever the quotient $N_{i}/ N_{i+1}$ is a pro-$\pi_1$ group. Since $ph(G)=\infty$, we have $ph(H_2)=\infty$. Therefore there exists a pro-$\pi_1'$ quotient $N_{k}/N_{k+1}$ such that $ph(N_{k}/N_{k+1})=\infty$. Set $R_2=N_{k+1}$. We now pass to the quotient $G/R_2$.

The finite $\pi_1$-group $A_1$ acts by conjugation on the pro-$\pi_1'$ group
$N_k/R_2$. By Proposition~\ref{thompson}, $ph(C_{N_k/R_2}(A_1))=\infty$. Thus, by Lemma~\ref{sta8} there exists a subgroup $ A_2$ containing $R_2$ such that $A_2/R_2$ is a finite subgroup of $C_{N_k/R_2}(A_1)$ with the property that $w(A_2/R_2)\neq1$. Set $\pi_2=\pi(A_2/R_2)$. We note that $\pi_1\cap\pi_2=\varnothing$ and $[A_1,A_2]\leq R_2$. Since $A_2/R_2 $ is finite, there is an open normal subgroup $J$ of $G$ such that $R_2\leq J$ and $A_2\cap J=R_2$. We set $H_3=J\cap H_2$; then $A_2\cap H_3=R_2\cap H_2=R_2$, so property 4 holds for $i=2$. Properties 1--3 also hold for $i=2$ by construction.

We now make the induction step. Suppose subgroups $A_1, \dots , A_{n-1}$, $R_1\leq \cdots\leq R_{n-1}$ and
$H_1\geq \cdots \geq H_{n-1}$ with properties 1--4 have been found.
For brevity we indicate with a bar the images of elements or subgroups in $G/R_{n-1}$. Since $R_1\leq R_2\leq \cdots$, by properties 1 and 2,  $\overline A_1, \dots ,\overline A_{n-1}$ are pairwise commuting finite subgroups of mutually coprime orders. In particular,
$\overline A=\overline A_1 \cdots \overline A_{n-1}$ is a finite $\pi$-group, where $\pi\subseteq \bigcup_{i=1}^{n-1}\pi_i$. Let $J$ be a normal open subgroup of $G$ containing $R_{n-1}$ such that $\overline J \cap \overline A=1$. Let $H_n=H_{n-1}\cap J$. Since the set $\pi$ is finite, it follows from Theorem~\ref{pro-p-solu} that the subgroup $\overline H_n$ has a finite series of closed characteristic subgroups $\{\overline N_{i}\}$ in which each factor  $\overline N_{i}/\overline N_{i+1}$ is either a pro-$\pi$ or a pro-$\pi'$ group. By Corollary~\ref{prosolu}, $ph(\overline N_{i}/\overline N_{i+1})$ is finite for each pro-$\pi$ factor quotient $\overline N_{i}/\overline N_{i+1}$. Taking into account that $ph(\overline G)= \infty$, we observe that also $ph(\overline H_n)= \infty$. Therefore there exists a pro-$\pi'$ factor  $T=\overline N_{k}/\overline N_{k+1}$ for some index $k$ such that $ph(T)=\infty$. Let $S$ and $R$ be the full pre-images of $\overline N_{k}$ and $\overline N_{k+1}$, so that $S/R$ is isomorphic to~$T$. The finite $\pi$-group $\overline A$ acts by conjugation on the pro-$\pi'$ group $S/R$. By Proposition~\ref{thompson},
$ph(C_{T}(\overline A))=\infty$. We therefore can use Lemma~\ref{sta8} to find a subgroup $A_n$ containing $R_n:=R$ such that
$A_n/R_n \le C_{S/R}(\overline A)$  and $A_n/R_n$ is a finite $\pi'$-group such that $w(A_n/R_n)\neq1$.
It is easy to see that then properties 1--4 are satisfied for $i=n$.
This completes the induction step. Thus, we have constructed
 an increasing chain of normal closed subgroups $1=R_1 \le R_2 \le \cdots$, a decreasing chain of normal open subgroups $G=H_1 \ge H_2 \ge \dots\ge H$, and subgroups $A_i\geq R_i$, $i=1,2,\dots $, with the required properties.

We now set $H=\bigcap_{n=1}^{\infty} H_i$; note that $H$ contains $\bigcup_{n=1}^{\infty} R_n$, since $H_{i+1}\geq R_i$ for every $i$ by property 4. Since $H$ is closed, we can consider the quotient $G/H$, where we use bars to denote images of subgroups. We see that $A_i\cap H=R_i$ for every $i$, since $A_i\cap H_{i+1}=R_i$ and $H_{i+1}\geq H\geq R_i$. Hence $\overline A_i=A_i/R_i$ for all~$i$.
By properties 1--4 it follows that the subgroups $\overline A_i$ pairwise commute, have pairwise coprime orders, and $w(\overline A_i)\ne 1$ for every~$i$.

Therefore the abstract group generated by the subgroups $\overline A_i$, $i=1,2\dots $, is isomorphic to the direct product of the $\overline A_i$. The closure of this direct product in the profinite topology is isomorphic to the Cartesian product of finite groups $\overline A_i$ of mutually coprime orders such that $w(\overline A_i)\neq 1$ for all~$i$. Obviously this group contains a $w$-value of infinite order. This yields a contradiction since all $w$-values in $G$ have finite order. The proof is now complete.
 \end{proof}

We can now finish the proof of Theorem~\ref{main}.

\begin{proof}[Proof of Theorem~\ref{main}.] We have a profinite group $G$ in which all pronilpotent subgroups generated by $w$-values are periodic; we need to show that $w(G)$ is locally finite. Obviously, $G\in\mathscr{Y}$. Recall that finite groups of odd order are soluble by the Feit--Thompson theorem \cite{fei-tho}. Combining this with Theorem~\ref{profi} (applied with $p=2$) we deduce that $G$ has a finite series of closed characteristic subgroups
\begin{equation}\label{e-riad}
G=G_0\geq G_1\geq \dots\geq G_s=1
\end{equation}
in which each factor either is prosoluble or is isomorphic to a Cartesian product of non-abelian finite simple groups. There cannot be infinitely many non-isomorphic non-abelian finite simple groups in a factor of the second kind, since this would give a $w$-value of infinite order. Indeed, by a result of Jones \cite{jones} any infinite family of finite simple groups generates the variety of all groups; therefore the orders of $w$-values cannot be bounded on such an infinite family.
Thus, we can assume in addition that each non-prosoluble factor in \eqref{e-riad} is
 isomorphic to a Cartesian product of \emph{isomorphic} non-abelian finite simple groups.  We use induction on~$s$.

 There is nothing to prove if $s=0$. Let $s\geq 1$. By induction, $w(G_1)$ is locally finite. Passing to the quotient $G/w(G_1)$ we can assume that $G_1$ is soluble. If $G/G_1$ is isomorphic to a Cartesian product of isomorphic non-abelian finite simple groups, then $G/G_1$ is locally finite and the result follows from Lemma~\ref{sta5}.

If $G/G_1$ is prosoluble, then so is $G$, and then $ph(G)$ is finite by Proposition~\ref{pi-prosol}. In this case, $w(G)$ is locally finite by Lemma~\ref{sta6}, as required.

\end{proof}

\section*{Acknowledgement}
The authors thank the referee for careful reading and helpful comments.

\bibliographystyle{line}

\end{document}